\documentclass[12pt]{article}
\usepackage{graphicx} 
\usepackage{amsmath,amssymb,amsthm}
\usepackage{tikz-cd}
\usepackage{comment}
\usepackage{hyperref}
\usepackage[left=1in,right=1in]{geometry}

\newcommand{\scat}[1]{\mathcal{#1}}  
\newcommand{\name}[1]{\mathsf{#1}}  

\renewcommand{\t}{\ensuremath{\mathsf{t}}}

\newcommand{\tT}{\ensuremath{{\t}:\TT\to\Ty}}
\newcommand{\Ty}{\ensuremath{\mathsf{T}}}
\newcommand{\TT}{\ensuremath{\dot{\mathsf{T}}}}

\newcommand{\catC}{\scat{C}}

\newcommand{\EE}{\scat{E}}

\newcommand{\I}{\mathbf{I}}

\newcommand{\too}{\longrightarrow}
\newcommand{\tto}{\rightrightarrows}

\newcommand{\mono}{\rightarrowtail}

\newcommand{\Set}{\name{Set}}

\newcommand{\Gpd}{\name{Gpd}}

\newcommand{\etalchar}[1]{$^{#1}$}
\newcommand{\op}[1]{{#1}^{\mathsf{op}}} 
\newcommand{\psh}[1]{{\Set}^{\op{#1}}} 

\newcommand{\pbmark}{\arrow[dr, phantom, "\lrcorner", very near start]}

  \newtheorem{theorem}{Theorem}[section]
  \newtheorem*{theorem*}{Theorem}
  \newtheorem{lemma}[theorem]{Lemma}
  \newtheorem{proposition}[theorem]{Proposition}

{\theoremstyle{definition}
  \newtheorem{definition}[theorem]{Definition}
  \newtheorem{example}[theorem]{Example}
  \newtheorem{remark}[theorem]{Remark}

}

\usepackage{listings}
\definecolor{keywordcolor}{rgb}{0.0, 0.1, 0.6}   
\definecolor{tacticcolor}{rgb}{0.0, 0.1, 0.6}    
\definecolor{commentcolor}{rgb}{0.4, 0.4, 0.4}   
\definecolor{symbolcolor}{rgb}{0.0, 0.1, 0.6}    
\definecolor{sortcolor}{rgb}{0.1, 0.5, 0.1}      
\definecolor{attributecolor}{rgb}{0.7, 0.1, 0.1} 

\newcommand{\leanlink}[1]{%
  \href{#1}{%
    \raisebox{-0.2ex}{\includegraphics[height=1em]{checkmark.pdf}}%
  }%
}

\title{Path Types in Algebraic Type Theory}
\author{Steve Awodey\thanks{Carnegie Mellon University} \and Joseph Hua\thanks{Imperial College London}}
\date{\today}
\begin{document}

\maketitle
\begin{abstract}
A new approach to the semantics of identity types in intensional Martin-L\"of type theory is proposed, assuming only a category with finite limits and an interval. The specification of \emph{extensional} identity types in the original presentation of natural models
paralleled that of the other type formers $\Sigma$ and $\Pi$, but the treatment of the \emph{intensional} case there was less uniform.  It was later reformulated to an account based on polynomials; here a further improvement in the style of the other type formers is achieved by employing an interval, in order to give a single pullback specification of a model with \emph{path types}.  The interval is also used to specify a (Hurewicz) fibration structure on the universe of the model. It is shown that the combination of these two conditions suffices to model the intensional identity rules, assuming only finite limits.  The addition of an interval also relates the current treatment to that of cubical type theory. 
\end{abstract}

\section{Introduction}

We propose a new approach to the semantics of identity types in intensional Martin-L\"of type theory, assuming only a category with finite limits and a bipointed, exponentiable object (an \emph{interval}). The treatment is quite simple and general, so it should apply in many different settings such as clans \cite{joyal2017}, model categories \cite{shulman2019, awodey2026}, path categories \cite{berg2018,BergMoerdijk2018}, categories with families \cite{dybjer2021}, and natural models \cite{awodey2016}. This new approach is adopted in the HoTTLean project \cite{hua2025}, simplifying the formalization in Lean.

The specification of \emph{extensional} identity types in the natural models of \cite{awodey2016} parallels that of the other type formers $\Sigma$ and $\Pi$ and is entirely satisfactory, but the treatment of the \emph{intensional} case there was less uniform. Following a suggestion of Richard Garner, it was updated in \cite{awodey2025} to a cleaner account based on polynomials; here a further improvement in the style of the other type formers is achieved by employing an interval, in order to give a single pullback specification of a model with \emph{path types}. The interval is also used to specify a (Hurewicz) fibration structure on the universe of the model. It is shown in Section \ref{sec:Interval} that the combination of these two conditions already suffices to model the usual intensional identity rules, assuming only finite limits and the exponentiability of the interval.

The addition of an interval also relates the current treatment to that of cubical type theory \cite{cohen2016}, as is explained in Section \ref{sec:cubical-Kan}. Specifically, we show that any classified type $A \to 1$ is in fact a cubical Kan complex; moreover, with a suitable (i.e.\ dependent) strengthening of the notion of path type, any classified type family $A \to X$ is a cubical Kan fibration.
We provide examples of path types in Section \ref{sec:examples} and conclude the paper by summarizing the Lean formalization
in Section \ref{sec:formalization}.

Natural models were originally defined in categories of presheaves and were then abstracted to the setting of a general locally cartesian closed category in \cite{awodey2025}. Here we only assume an ambient category with finite limits, while retaining the same approach to the semantics of type theory based on a \emph{classifier for types} $\tT$ (with chosen pullbacks), which we still refer to as a \emph{natural model} in the present, more general setting.

\paragraph{Acknowledgement}
We have benefitted from discussions with our collaborators on the HoTTLean project: Mario Carneiro, Sina Hazratpour, Wojciech Nawrocki, Spencer Woolfson, and Yiming Xu, as well as with Reid Barton.  Evan Cavallo noticed a critical error in an earlier version of this paper and proposed the repair now used in Section \ref{sec:cubical-Kan}; we are grateful for his generous advice.  This material is based upon work supported by the Air Force Office of Scientific Research under MURI award number FA9550-21-1-0009.

\section{Path types}\label{sec:Interval}

\begin{definition}\label{def:Interval}

By an \emph{interval} in a finite limit category $\EE$ we simply mean a bipointed, exponentiable object $d_0, d_1:1\rightrightarrows \I$.  In terms of such an interval, we then define further:
\begin{enumerate} 
\item For every object $A$, there is a \emph{pathobject} factorization of the diagonal $\delta:A \to A\times A$, obtained by exponentiating $A$ by $1\rightrightarrows \I \to 1$, 
\[\begin{tikzcd}
	A \ar[r,"\rho"] \ar[rd, swap,"\delta"] & A^\I \ar[d,"{\langle \varepsilon_0, \varepsilon_1\rangle}"]\\
	& A \times A
 \end{tikzcd} 
\]
where we write $\rho = A^{!_\I}$, and $\varepsilon_0 = A^{d_0}$, and $\varepsilon_1 = A^{d_1}$.

\item Similarly, and abusing notation slightly, for any $A\to X$ regarded as an object in the slice category $\EE/_{\!X}$ we define the \emph{relative pathobject factorization}
\[\begin{tikzcd}
	A \ar[r,"\rho"] \ar[rd, swap,"\delta"] & A^\I \ar[d,"{\langle \varepsilon_0, \varepsilon_1\rangle}"]\\
	& A \times_{\!X} A
 \end{tikzcd} 
\]
to be the pathobject factorization in $\EE/_{\!X}$ with respect to the pulled-back interval $1_X \rightrightarrows {!_X^*}\I \to 1_X$, where we are using the pullback functor ${!_X^*} : \EE\to\EE/_{\!X}$ along $!_X : X \to 1$.
\end{enumerate}
We call $\varepsilon := \langle \varepsilon_0, \varepsilon_1\rangle : A^I \to A \times_{\!X} A$ the \emph{(relative) pathobject} of $A \to X$,
and $\rho : A \to A^\I$ its \emph{reflexivity path}.
\end{definition}
   
\begin{lemma}
   For any object $A\to X$ over any base $X$, the relative pathobject factorization
\[\begin{tikzcd}
	A \ar[r,"\rho"] \ar[rd, swap,"\delta"] & A^\I \ar[d,"\varepsilon"]\\
	& A \times_{\!X} A\,,
 \end{tikzcd} 
\]
is stable under pullback along any map $f : Y\to X$, in the sense that the factorization pulls back to the relative pathobject factorization over $Y$ of the pullback $f^*A\to Y$, resulting in a canonical isomorphism over $Y$,
\[
f^*(A^\I)\, \cong\,  (f^*A)^\I\,.
\]
\qed
\end{lemma}

\begin{definition}\label{def:pathtypes}
A natural model $\tT$ in a finite limit category $\EE$ with an interval will be said to \emph{have path types} if there are structure maps $(\mathsf{Path},\mathsf{path})$ making a pullback square,
\begin{equation}\label{diag:pathtypes}
\begin{tikzcd}
	\TT^\I \ar[r,"\mathsf{path}"] \ar[d,swap,"\varepsilon"] &  \TT \ar[d, "\t"]\\  
	\TT \times_\Ty \TT \ar[r, swap,"\mathsf{Path}"] & \Ty
 \end{tikzcd}
 \end{equation}
 where $\varepsilon :\TT^\I \to \TT \times_\Ty \TT $ is the relative pathobject of $\tT$ over $\Ty$.
 \end{definition}

In order to show that a natural model with path types also has intensional Identity types in the sense of \cite[Definition 11]{awodey2025}, we require an additional condition on the map $\tT$, namely that it should be a ``Hurewicz fibration'' in the following sense.

\begin{definition}\label{def:Hfibration}
A map $f: Y\to X$ will be called a \emph{Hurewicz fibration} (with respect to an interval $d_0, d_1:1\rightrightarrows \I$),  if either one of the following equivalent conditions holds:
\begin{enumerate}
    \item The map $f$ has the right lifting property with respect to the
    $0$-end inclusion of every cylinder $Z\times d_0 : Z\times 1\to Z\times \I$.
    In detail, given any object $Z$ and maps $y$ and $h$ as indicated below, there exists a diagonal filler $\tilde{h}$ making the following diagram commute.
    \begin{equation}
    \begin{tikzcd}
    Z\times 1\ar[d,swap, "{Z\times d_0}"] \ar[r, "y"] &  {Y} \ar[d, "f"]\\
    Z\times \I \ar[r,swap,"h"]\ar[ru,swap,dotted,"\tilde{h}"]   & {X}
    \end{tikzcd}
    \end{equation}
    We may regard $h : Z\times \I \to X$ as a homotopy between the maps $h_0, h_1 : Z\to X$ obtained by composing it with the two ends of the cylinder  $Z\times 1 \rightrightarrows Z\times \I$, and the lift $\tilde{h} : Z\times \I \to Y$ as one with the specified $0$-end $y= \tilde{h}_0$.  (Note that we do not need to also require a lift for the $1$-end inclusion $Z\times d_1 : Z\times 1\to Z\times \I$.)
    \item The following diagram is a weak pullback.
    \begin{equation}\label{diag:hfibwpb}
    \begin{tikzcd}
	   Y^\I \ar[d, swap, "{f^\I}"] \ar[r, "{\varepsilon_0}"] & Y \ar[d, "f"] \\  
	   X^\I \ar[r, swap,"\varepsilon_0"] & X
    \end{tikzcd}
    \end{equation}
    This means that the comparison map to the actual pullback has a section $\ell : X^\I \times_{\!X} Y \to Y^\I$, as indicated below.
    \begin{equation}\label{diag:hfibequiv}
    \begin{tikzcd}
	   Y^\I \ar[rdd, bend right,swap, "f^\I"] \ar[rd] \ar[rrd, bend left, "\varepsilon_0"] && \\
	   & X^\I \times_{\!X} Y  \ar[lu, bend right = 20, dotted,swap, "\ell"]  \ar[d] \ar[r]  & Y \ar[d, "f"]\\  
	   & X^\I \ar[r, swap,"\varepsilon_0"] & X
    \end{tikzcd}
    \end{equation}
\end{enumerate}
 \end{definition} 
 \begin{proof} $(1) \iff (2)$ (in presheaves on $\EE$, if need be for colimits):
 
In terms of the so-called ``Leibniz adjunction'' ${\otimes} \dashv {\Rightarrow}$ (see \cite{JT2007}), the comparison map ${\langle f^\I, \varepsilon_0 \rangle} : Y^\I \to X^\I \times_{\!X} Y$ is the ``pullback-hom'' $d_0\Rightarrow f$, in the notation of \cite{awodey2026}.

Now, an arbitrary map $g : A\to B$ has a section just if it lifts on the right against $0\to Z$, for all objects $Z$ (one can take $Z=B$). Thus ${d_0\Rightarrow f}$ has a section just if $0_Z \pitchfork (d_0\Rightarrow f)$ for all $Z$, which is equivalent by the adjunction to $(0_Z \otimes d_0) \pitchfork f$ for all $Z$.  But we have $0_Z \otimes d_0 = Z \times d_0 : Z\to Z\times \I$.
 \end{proof}
 
A section $\ell : X^\I \times_{\!X} Y \to Y^\I$ as in \eqref{diag:hfibequiv} may be regarded as a ``lifting operation'' that takes a path $p : x_0 \leadsto x_1$ in $X$, and a point $y_0\in Y$ over $x_0$, to a lifted path $\tilde{p} : y_0 \leadsto y$ in $Y$ lying over $p$.  

\begin{definition}\label{def:normalHfibration}
A Hurewicz fibration $f : Y\to X$ is called \emph{normal} if it comes with a lifting operation $\ell : X^\I \times_{\!X} Y \to Y^\I$, as in (2) of Definition \ref{def:Hfibration}, that moreover takes a reflexivity path in $X$ to a reflexivity path in $Y$, as indicated below.
\begin{equation}\label{diag:normal}
\begin{tikzcd}
Y \ar[rr,"\rho"] && Y^\I \\
 X\times_{\!X} Y \ar[u, "p_2"] \ar[rr, swap,"{\rho \times_{\!X} Y}"] && X^\I \times_{\!X} Y  \ar[u,swap,"\ell"]
 \end{tikzcd}
 \end{equation}
 \end{definition}
\smallskip

Now consider the following \emph{axioms} on a natural model $\tT$ in a finite limit category $\EE$ with an interval $d_0, d_1:1\rightrightarrows \I$.
\begin{itemize}
\item[(A1)] $\tT$ has path types, as in Definition \ref{def:pathtypes}.
\item[(A2)] $\tT$ is a normal Hurewicz fibration, as in Definition \ref{def:normalHfibration}.
\end{itemize}

\begin{lemma} \label{lemma:connection-on-U}
Assuming axioms (A1) and (A2), the map $\tT$ has a \emph{connection}: a map $\chi : \TT^\I \to (\TT^\I)^\I \cong \TT^{\I\times\I}$ (over $\Ty$) making the following commute. (An elementary illustration is given in Diagram \ref{diag:2boxfilling}.)
\begin{equation}\label{diag:connection}
\begin{tikzcd}
	& \TT^\I  \\
\TT^\I \ar[r, swap,"\chi"] \ar[d,swap, "\varepsilon_0"] \ar[ru, "{=}"] & \TT^{\I\times\I} \ar[d, "\varepsilon_0"] \ar[u,swap, "\varepsilon_1"]  \\ 
\TT  \ar[r, swap,"\rho"] & \TT^\I 
 \end{tikzcd}
 \end{equation}
Moreover, the connection $\chi$ is \emph{normal} in the sense that $\chi\rho = \rho\rho$\,,
\begin{equation}\label{diag:normalconnection}
\begin{tikzcd}
\TT^I \ar[r,"\chi"] & \TT^{\I\times\I}  \\ 
\TT  \ar[u, "\rho"] \ar[r, swap,"\rho"] & \TT^\I.\ar[u,swap, "\rho"] 
 \end{tikzcd}
 \end{equation}
\end{lemma}
\begin{proof}
We can use ``box-filling'' to construct the connection as the transpose of a certain diagonal filler
\[
\tilde{\chi}  : \TT^\I \times \I \to \TT^\I
\]
as follows.  Given any pathspace $A^\I \to A\times A$ that is a Hurewicz fibration, consider a lifting problem of the form: 
\begin{equation}\label{diag:liftforHpathspace}
\begin{tikzcd}
1\ar[d,swap,"{\delta_0}"] \ar[r, "r"] &  {A^\I} \ar[d]\\
\I \ar[r,swap,"{\langle q, p \rangle}"]\ar[ru,dotted,"\tilde{b}"]   & {A\times A}
\end{tikzcd}
\end{equation}
This can be regarded as filling an ``open box'' $(p,q,r)$ in $A$, to give a 2-cube $b: \I\times \I \to A$, which can be depicted schematically as follows:
\begin{equation}\label{diag:2boxfilling}
\begin{tikzcd}
q_0\ar[d,swap,"q"] \ar[r, "r"] &  p_0 \ar[d, "p"]\\
q_1 \ar[r, dotted] \ar[ru,phantom, "b"]  & p_1
\end{tikzcd}
\end{equation}
Now consider the case $q = \rho_{p_0} = r$ and take $\chi_p : \I \to A^\I$ to be the resulting cube: 
\begin{equation}\label{diag:2boxfilling1}
\begin{tikzcd}
p_0\ar[d,swap,"="] \ar[r, "="] &  p_0 \ar[d, "p"]\\
p_0 \ar[r, dotted] \ar[ru,phantom, "\chi_p"]  & p_1
\end{tikzcd}
\end{equation}

We apply this to the case where $A = \TT$ (over $\Ty$), with object of parameters $\TT^\I$, to obtain the following:
\begin{equation}\label{diag:liftforHpathspace1}
\begin{tikzcd}
\TT^\I\times_{\!\Ty} 1\ar[d,swap,"{\TT^\I \times_{\!\Ty}\,\delta_0}"] \ar[r, "r"] &  {\TT^\I} \ar[d]\\
\TT^\I\times_{\!\Ty} \I \ar[r,swap,"{\langle q, p \rangle}"]\ar[ru,dotted,swap,"\tilde{\chi}"]   & {\TT \times_{\!\Ty} \TT}
\end{tikzcd}
\end{equation}
The maps $p,q,r$ are now as follows:
\begin{align*}
p &= \mathsf{eval}: \TT^\I\times_{\!\Ty} \I \to \TT \\
q &= \mathsf{eval} \circ (\rho \varepsilon_0 \times_{\!\Ty} \I): \TT^\I\times_{\!\Ty} \I \to \TT\\
r &= \rho\varepsilon_0 : \TT^\I \to \TT^\I
\end{align*}
Transposing provides the desired map $\chi : \TT^\I \to \TT^{\I\times\I}$ with evaluations $\chi(p)_0 = \rho p_0$ and $\chi(p)_1 = p$ for all $p:\TT^\I$. 
\end{proof}

The usual Id-Elim rule of intensional Martin-L\"of type theory now holds for any type family $A \to X$ that is classified by $\tT$, as follows.   By (A2), $A \to X$ is Hurewicz since it is a pullback of $\tT$.  The path type $A^\I\to A\times_X A$ is also Hurewicz, since it is a pullback of $\TT^\I \to \TT\times_\Ty \TT$, which is Hurewicz by (A1).  By Lemma \ref{lemma:connection-on-U} there is also a (normal) connection on $A\to X$ (again since it is a pullback of a map with such a connection).  

Consider an Id-elimination problem as follows:
\begin{equation}\label{diag:Id-Elim_with_pathtypes}
\begin{tikzcd}
A \ar[d,swap, "\rho"] \ar[r, "c"] &  {C} \ar[d]\\
A^\I\ar[r, swap,"="]\ar[ru,swap,dotted,"j"]   & {A^\I}
\end{tikzcd}
\end{equation}
where the type family $C\to A^\I$ is a pullback of $\tT$, and therefore a Hurewicz fibration, with (normal) lifting operation $\ell : (A^\I)^\I \to C^\I$.  

We argue first with elements to give the idea of the proof, before giving a diagrammatic version of the same argument.  Thus take any $p : A^\I$, which is a path $p:p_0 \leadsto p_1$ in $A$.  Applying the connection on $A$ (resulting from Lemma \ref{lemma:connection-on-U}),  we obtain a (higher) path $\chi_p : \rho p_0 \leadsto p$ in $A^\I$.  Since the outer square of \eqref{diag:Id-Elim_with_pathtypes} commutes, for every $p_0:A$ the term $cp_0:C$ lies over $\rho p_0$, thus we can lift $\chi_p$ to a path $\ell\langle\chi_p, cp_0\rangle  : cp_0 \leadsto \tilde{p}$ in $C$, with the endpoint $\tilde{p}$ over $p$.  We then set 
\[
jp := \tilde{p} = \varepsilon_1\ell\langle\chi_p, cp_0\rangle \,,
\]
which plainly makes the bottom triangle of \eqref{diag:Id-Elim_with_pathtypes} commute.  Observe that in the case $p = \rho a$ for $a:A$, we have $\chi_{\rho a} = \rho_{\rho a}$ since $\chi$ is normal, and then for the lift we have $\ell\langle\chi_{\rho a}, ca\rangle  = \ell\langle\rho_{\rho a}, ca\rangle = \rho ca$, since $\ell$ is normal. Whence $j\rho{a} = \varepsilon_1\rho ca = ca$, as required for the top triangle of \eqref{diag:Id-Elim_with_pathtypes} to commute.

We summarize the result as follows (cf.\ \cite{awodey2018} for a related result).
\begin{proposition} \label{prop:Id-Elim}
Let $\tT$ be a natural model in a finite limit category $\EE$ with an interval $1\rightrightarrows\I$, and assume axioms (A1) and (A2) above.  Then for all $A\to X$ classified by $\tT$, the path type $A \to A^\I \to A\times_{\!X} A$ validates the usual identity type rules of intensional type theory (as stated in diagram \eqref{diag:Id-Elim_with_pathtypes}).
\end{proposition}

\begin{proof}
The foregoing ``elementary'' proof is depicted in the following diagram:
\begin{equation}\label{diag:Id-Elim_diagram}
\begin{tikzcd}
A \ar[ddd,swap, "\rho"] \ar[rr, "c"] &&  {C} \ar[dd, "="]\\
	& C^\I \ar[ru, "\varepsilon_1"] & \\
	& * \ar[u, near start, swap, "\ell"] \ar[d] \ar[r] \pbmark & C \ar[d] \\
A^\I \ar[rr, bend right = 5ex, swap,"="] \ar[rruuu,dotted,near start, "j"] \ar[ur] \ar[r, "\chi"] 
	& (A^\I)^\I  \ar[r, "\varepsilon_0"] & {A^\I}
\end{tikzcd}
\end{equation}
Where $* = (A^\I)^\I \times_{A^\I} C$ is the indicated pullback, and the unlabelled map into it is the pair
\[
\langle \chi, c\varepsilon_0 \rangle : A^\I \too *\,.
\]
We leave the diagram chase (just sketched) to the reader.
\end{proof}

Since the map $j = \varepsilon_1 \circ\ell \circ \langle \chi, c\varepsilon_0 \rangle $ is defined from others that are stable under pullback (themselves being pullbacks of a universal instance, defined from structure on $\tT$), the substitution condition with respect to a change of context $\sigma : X' \to X$ will obtain; that is we shall have:
\[
(j^c)_\sigma = j^{c_\sigma} : C_\sigma\,.
\]
From this, it follows that we can determine a weak pullback structure map $J$ as required in \cite[Definition 11]{awodey2025}, relating the two polynomial functors associated with $\tT$ and $\TT^\I \to \Ty$, evaluated at $\tT$.

\begin{remark}[Tiny interval]
A simpler formulation of this ``substitution condition'' is available in the case when the interval $\I$ is a \emph{tiny} object, i.e.\ one for which the functor $(-)^\I$ of exponentiation has a right adjoint \emph{root} functor $(-)_\I$, as obtains e.g.\ in cubical sets, and in presheaves over any finite product category $\catC$ of ``contexts''.   In that case, let us reformulate the elimination diagram \eqref{diag:Id-Elim_with_pathtypes} in the equivalent form: 

\begin{equation}\label{diag:Id-Elim_with_pathtypes2}
\begin{tikzcd}
A \ar[d,swap, "\rho"] \ar[r, "c"] &  \TT \ar[d, "\mathsf{u}"]\\
A^\I\ar[r, swap,"C"]\ar[ru,swap,dotted,"j"]   & {\Ty}
\end{tikzcd}
\end{equation}
Then, writing $\rho = A^{!} : A \to A^\I$, for $!: \I \to 1$, we can transpose across the adjunction $(-)^\I \dashv (-)_\I$ to obtain the following equivalent problem:
\begin{equation}\label{diag:Id-Elim_with_pathtypes3}
\begin{tikzcd}
A \ar[ddr, bend right, swap, "{\tilde{C}}"] \ar[rd,swap,dotted,"\tilde{j}"] \ar[rrd, bend left, "{c}"] && \\
&\TT_\I \ar[d,swap, "\mathsf{u}_\I"] \ar[r, "\TT_!"] &  {\TT} \ar[d, "\mathsf{u}"]\\
&\Ty_\I  \ar[r, swap,"\Ty_!"]   & {\Ty}
\end{tikzcd}
\end{equation}
Note that we are using the fact that the terminal object $1$ is also tiny,
where $(-)_1= \mathsf{id}$ is the identity functor. 
The natural transformation between left adjoints $\rho = (-)^! : (-)^1 \to (-)^\I$ gives rise to a conjugate transformation between their right adjoints $(-)_! : (-)_\I \to (-)_1$.

The problem \eqref{diag:Id-Elim_with_pathtypes3} can then be reformulated without reference to $A, \tilde{C}, c$ as stating that the inner square is a weak pullback, or, again equivalently, that there exists a section $\tilde{\mathsf{J}}$ of the comparison map ${\langle \mathsf{u}_\I, \TT_! \rangle}$ to the actual pullback,
\begin{equation}\label{def:wps2}
\begin{tikzcd}
\TT_\I  \ar[rr] && \ar[ll, bend right, dotted, "\tilde{\mathsf{J}}" description] \Ty_\I \times_{\Ty} \TT 
\end{tikzcd}
\end{equation}
This condition is obviously independent of any particular maps $A : X \to \Ty$, etc., into the object $\tT$, and therefore evidently satisfies the strict rule of coherence under substitution.
Comparing the condition \eqref{def:wps2} with its counterpart in \cite[Definition 11]{awodey2025},
we have achieved a simplification in replacing the polynomial functors $P_\mathsf{u}, P_q : \hat{\catC}\to\hat{\catC}$ by the root functor $(-)_\I:\hat{\catC}\to\hat{\catC}$.
\end{remark}

\section{Examples} \label{sec:examples}

\begin{example}\label{example:sets}
  When $\I = 1$,
  (A2) holds for any map,
  and (A1) corresponds to the pullback square defining \emph{extensional} Identity types (cf.~\cite{awodey2025}):
  \begin{equation}
  \begin{tikzcd}
	\TT \ar[r,"\mathsf{refl}"] \ar[d,swap,"\delta"] &  \TT \ar[d, "\t"]\\  
	\TT \times_\Ty \TT \ar[r, swap,"\mathsf{Eq}"] & \Ty
  \end{tikzcd}
  \end{equation}

  In particular when $\EE = \psh{C}$ and $\tT$ is a Hofmann-Streicher universe \cite{hofmann1997},
  such a pullback square can always be constructed (see \cite[\S{7.1}]{awodey2026}).
\end{example}

\begin{example}\label{example:groupoids}
  Take $\EE = \Gpd$ to be the category of groupoids
  and $\I$ to be the interval groupoid
  (also called the ``walking isomorphism'').
  Then (A2) holds just if $\tT$ is an isofibration,
  and both (A1) and (A2) hold if we take $\tT$ to be
  (the core of) the forgetful functor from small pointed groupoids
  to small groupoids,
  as is the approach in \cite{hua2025}.
\end{example}

\begin{example} \label{example:cSet}
  Take $\EE = \mathsf{cSet}$ to be the category of
  Cartesian cubical sets, $\I = y(\Box_1)$
  to be the cubical interval,
  and $\tT$ to be a classifier for certain cubical fibrations,
  such as the classifier for small unbiased fibrations
  (henceforth just ``fibrations'') in
  \cite[Proposition 7.17]{awodey2026}.
  Then the classifier itself is a fibration,
  and therefore satisfies (A2).
  To construct path types (A1),
  suppose first that $A \to X$ is classified by $\tT$.
  Then in particular $A \to X$ is a (small) fibration,
  from which it follows that the relative pathobject
  $A^I \to A \times_{\!X} A$ is also a (small) fibration,
  and therefore is classified.
\[\begin{tikzcd}
	{A^I} & \TT \\
	{A \times_{\!X} A} & \Ty
	\arrow[from=1-1, to=1-2]
	\arrow[from=1-1, to=2-1]
	\arrow["\lrcorner"{anchor=center, pos=0.125}, draw=none, from=1-1, to=2-2]
	\arrow["\t", from=1-2, to=2-2]
	\arrow[from=2-1, to=2-2]
\end{tikzcd}\]
  We can then specialize this to the case where $A \to X$ is the universe $\t$ itself (the universe classifies itself trivially),
  resulting in the pullback square required for (A1).
\end{example}

The construction in Example \ref{example:cSet}
should generalize to any monoidal model category
with a classifier for small fibrations.
In fact, Examples \ref{example:sets} and \ref{example:groupoids} are also of this form.
The model structure for \ref{example:sets} has all maps as
cofibrations and fibrations,
and isomorphisms as weak equivalences.
The model structure for \ref{example:groupoids} is the so-called 
canonical model structure on groupoids.

\section{Cubical Kan fibrations}\label{sec:cubical-Kan}

Let $\EE$ be a finite limit category with an interval $1\rightrightarrows\I$ and $\tT$ a natural model in $\EE$.   It will be convenient to also have (stable) finite colimits in $\EE$, for example as when $\EE = \widehat\catC$ is presheaves on a category $\catC$ of ``contexts''.  The powers $\I^n$ of the interval give rise to a cubical structure in $\EE$, with respect to which one can consider Kan-style ``box-filling'' conditions, generalizing the path-lifting condition in Definition \ref{def:Hfibration} from an endpoint inclusion $1\to \I$ to the inclusion of an ``open box'' $\sqcup^n \mono \I^n$, i.e.\  a subobject defined as the join of all but one $(n-1)$-dimensional faces $\I^{n-1} \mono \I^n$. 

%
%

As motivation for the following, consider the following observation from \cite[Proposition 14]{awodey2018}.

\begin{proposition}\label{lemma:boxfilling}
 For any object $A$ in $\EE$ the following are equivalent:
 \begin{enumerate}
\item The pathobject $A^\I \to A\times A$ has $n$-box filling.
\item The object $A$ has $(n+1)$-box filling.
\end{enumerate}
\end{proposition}

\begin{proof}
To reason about filling open boxes, it is convenient to use the so-called Joyal-Tierney calculus (or ``Leibniz adjunction'') on the arrow category $\EE^\downarrow$, relating the pushout-product $a\otimes b$ and the pullback-hom $c\Rightarrow d$. Recall from e.g.\ \cite{JT2007} that for all arrows $a,b,c$ in $\EE$, 
 \[
 (a\otimes b)\, \pitchfork\, c\quad \text{iff}\quad a\, \pitchfork\, (b \Rightarrow c)\,,
 \] 
where $x\pitchfork y$ is the relation of \emph{weak orthogonality}, meaning that every commutative square from $x$ to $y$ has a diagonal filler.   In these terms, observe that the (upper, say) open boxes may be described as pushout-products of the form $\sqcup^{n+1} = \partial\I^{n} \otimes d_0$, while the boundary $\partial\I^{n} \mono \I^{n}$ itself satisfies $\partial\I^{n} = \partial\I \otimes \partial\I^{n-1}$ (and so, generally, $\partial\I^m \otimes \partial\I^n = \partial\I^{m+n}$).  Finally, we have $(A^\I \to A\times A) = (\partial\I\Rightarrow A)$, where $A$ stands for the unique map $A\to 1$. 

In these terms, we therefore wish to show that for any $n$ and open $n$-box inclusion $\sqcup^n \mono \I^{n}$, we have $(\sqcup^n \to \I^n)\ \pitchfork\ (A^\I \to A\times A)$ 
if and only if $(\sqcup^{n+1} \to \I^{n+1})\ \pitchfork\ A$.  But, indeed,
\begin{align*}
(\sqcup^n \to \I^n)\ \pitchfork\ (A^\I \to A\times A)\quad
\text{iff}\quad &(\partial\I^{n-1} \otimes d_0)\ \pitchfork\ (\partial\I\Rightarrow A)\\
\text{iff}\quad &(\partial\I^{n-1} \otimes d_0)\otimes \partial\I\ \pitchfork\ A\\
\text{iff}\quad &(\partial\I\otimes\partial\I^{n-1} ) \otimes d_0\ \pitchfork\ A\\
\text{iff}\quad &(\partial\I^n \otimes d_0)\ \pitchfork\ A\\
\text{iff}\quad &(\sqcup^{n+1} \to \I^{n+1})\ \pitchfork\ A\,.
\end{align*}
\end{proof}

In order to now generalize the foregoing argument from objects $A$ to maps $A\to X$, we introduce the following ``dependent'' notion of path-types,\footnote{due to Evan Cavallo} as a strengthening of Definition \ref{def:pathtypes}, in which the pullback-hom $\partial\I\Rightarrow\t$ replaces the (relative) exponential $\varepsilon = \t^{\partial\I}$ in diagram~\eqref{diag:pathtypes}.

\begin{definition}\label{def:deppathtypes}
A natural model $\tT$ in a finite limit category $\EE$ with an interval $1\rightrightarrows \I$ will be said to have \emph{dependent path types} if there are structure maps $(\mathsf{dPath},\mathsf{dpath})$ making a pullback square,
\begin{equation}\label{diag:dpathtypes}
\begin{tikzcd}
	\TT^\I \ar[r,"\mathsf{dpath}"] \ar[d,swap,"\partial\I\Rightarrow \t"] \pbmark &  \TT \ar[d, "\t"]\\  
	\Ty^\I \times_{\Ty\times\Ty} (\TT\times\TT)  \ar[r, swap,"\mathsf{dPath}"] & \Ty
 \end{tikzcd}
 \end{equation}
 in which the map $\partial\I\Rightarrow \t$ 
 is the pullback-hom of the maps $\partial\I : 1+1 \to\I$ and $\tT$.
 \end{definition}

%

Note that this condition implies the existence of path types, as stated in Definition \ref{def:pathtypes}.
We also note the following simple but crucial \emph{observation}:
given a pullback square
 \begin{equation*}\label{diag:def_HFib}
\begin{tikzcd}
A' \ar[d,swap, "\alpha"] \ar[r] \pbmark &  B' \ar[d, "\beta"]\\
A  \ar[r]  & B,
\end{tikzcd}
\end{equation*}
for any map $\gamma : C'\to C$, we have that
$\gamma \pitchfork \beta\ \text{implies}\ \gamma \pitchfork \alpha$.

%

\begin{proposition}\label{prop:Kanfilling}
Let $\tT$ be a normal Hurewicz fibration with dependent path types, as in Definition \ref{def:deppathtypes}.  Then every map $A\to X$ classified by $\tT$ has $n$-box filling for all $n>0$.
\end{proposition}

\begin{proof}
By the foregoing observation, it suffices to prove that $\tT$ itself has $n$-box filling for all $n>0$, since then any pullback of it will have the same.  As an Hurewicz fibration, $\tT$ already has $1$-box filling.   Suppose it has $n$-box filling; then by the observation, so does the \emph{dependent} pathobject  $(\partial\I\Rightarrow t)$, since it is a pullback of $\tT$ by \eqref{diag:dpathtypes}.  Then, reasoning just as in Proposition \ref{lemma:boxfilling} by the Leibniz adjunction, $\tT$ itself therefore has $(n+1)$-box filling. Thus, by induction, $\tT$ has $n$-box filling for all $n>0$.
\end{proof}

In this sense, all of the ``type families'' $A\to X$ classified by $\tT$ are \emph{cubical Kan fibrations}. 
%
%
%

\section{Formalization in Lean}\label{sec:formalization}

A version of Proposition \ref{prop:Id-Elim} is formalized as part of the HoTTLean project, 
using path types for the construction of identity types.
Here we give an overview of the formalization,
which is available at the following stable branch of the HoTTLean github repository:

\begin{center}
  \href{https://github.com/sinhp/HoTTLean/blob/TYPES2026/HoTTLean/Model/Unstructured/Hurewicz.lean}{
    \textsf{github.com/sinhp/HoTTLean/blob/TYPES2026}}
\end{center}

A key difference between the formalization and the exposition in this paper is that for the sake of simplicity,
only item (1) in Definition \ref{def:Hfibration} is used,
namely the definition of a Hurewicz fibration ``on the left'' of the adjunction $\I \times (-) \vdash (-)^\I$.
The main benefit of working ``on the left''
is that the definitions hold in a more general setting,
without needing to handle exponentiability conditions.
On the other hand, we will not be able to work \emph{universally}:
the Hurewitz lifting structure will be required to be uniform
(stable under substitution),
whereas working with a single lifting operation on the natural model
$\ell : \Ty^\I \times_{\!X} \TT \to \TT^\I$
instead guarantees a uniform lifting structure.
This prevents us from defining Path types in the style of 
natural models:
instead of describing (A1) as a pullback,
the formalization unpacks the universal property of the pullback in elementary terms
similar to categories with families \cite{dybjer2014}.

For the benefits of abstraction,
the formalization replaces the interval object $\I$ with an endofunctor
$\mathcal{I}: \EE \to \EE$ replacing $\I \times (-)$,
equipped with natural transformations
$\delta_0, \delta_1 : 1 \tto \mathcal{I}$
replacing $\delta_i \times (-) : 1 \times (-) \tto \I \times (-)$,
$\pi : \mathcal{I} \to 1$,
and \lstinline{symm} $ : \mathcal{I} \circ \mathcal{I} \to \mathcal{I} \circ \mathcal{I}$ satisfying
some equations.
Accordingly, the rest of the proof of \ref{prop:Id-Elim} is also phrased in this style.
This is captured in the following Lean definition
\lstinline{Cylinder}.

\begin{lstlisting}
structure Cylinder (Ctx : Type u) [Category.{v} Ctx] where
  (I : Ctx ⥤ Ctx)
  (δ0 δ1 : $\mathbf{1}$ _ ⟶ I)
  ($\pi$ : I ⟶ $\mathbf{1}$ _)
  (δ0_$\pi$ : δ0 ≫ $\pi$ = $\mathbf{1}$ _)
  (δ1_$\pi$ : δ1 ≫ $\pi$ = $\mathbf{1}$ _)
  (symm : I ⋙ I ⟶ I ⋙ I)
  (symm_symm : symm ≫ symm = $\mathbf{1}$ _)
  (whiskerLeft_I_δ0_symm : whiskerLeft I δ0 ≫ symm = whiskerRight δ0 I)
  (whiskerLeft_I_δ1_symm : whiskerLeft I δ1 ≫ symm = whiskerRight δ1 I)
  (symm_π_π : symm ≫ whiskerLeft I π ≫ π = whiskerLeft I π ≫ π)
\end{lstlisting}

A universe is formalized as a morphism \lstinline{tp}
with chosen pullbacks \lstinline{ext A} for every map
\lstinline{A} into the base.

\begin{lstlisting}
structure UnstructuredUniverse (Ctx : Type u) [Category Ctx] where
  Tm : Ctx
  Ty : Ctx
  tp : Tm ⟶ Ty
  ext {Γ : Ctx} (A : Γ ⟶ Ty) : Ctx
  disp {Γ : Ctx} (A : Γ ⟶ Ty) : ext A ⟶ Γ
  var {Γ : Ctx} (A : Γ ⟶ Ty) : ext A ⟶ Tm
  disp_pullback {Γ : Ctx} (A : Γ ⟶ Ty) :
    IsPullback (var A) (disp A) tp A
\end{lstlisting}

Given a universe, 
the definition of Path types
(corresponding to Definition \ref{def:pathtypes}) can be stated.

\begin{lstlisting}
structure PathType (U : UnstructuredUniverse Ctx) where
  (Path : ∀ {Γ} {A : Γ ⟶ U.Ty} (a0 a1 : Γ ⟶ U.Tm), (a0 ≫ U.tp = A) → a1 ≫ U.tp = A →
    (Γ ⟶ U.Ty))
  (Path_comp : ∀ {Γ Δ} (σ : Δ ⟶ Γ) {A : Γ ⟶ U.Ty} (a0 a1 : Γ ⟶ U.Tm)
    (a0_tp : a0 ≫ U.tp = A) (a1_tp : a1 ≫ U.tp = A),
    Path (A := σ ≫ A) (σ ≫ a0) (σ ≫ a1) (by simp [a0_tp]) (by simp [a1_tp]) =
    σ ≫ Path a0 a1 a0_tp a1_tp)
  (path : ∀ {Γ} {A : Γ ⟶ U.Ty} (p : cyl.I.obj Γ ⟶ U.Tm),
    p ≫ U.tp = cyl.π.app Γ ≫ A → (Γ ⟶ U.Tm))
  (path_comp : ∀ {Γ Δ} (σ : Δ ⟶ Γ) {A : Γ ⟶ U.Ty}
    (p : cyl.I.obj Γ ⟶ U.Tm) (p_tp : p ≫ U.tp = cyl.π.app Γ ≫ A),
    path (A := σ ≫ A) ((cyl.I.map σ) ≫ p) (by simp [p_tp]) =
    σ ≫ path p p_tp)
  (path_tp : ∀ {Γ} {A : Γ ⟶ U.Ty} (a0 a1 : Γ ⟶ U.Tm) (p : cyl.I.obj Γ ⟶ U.Tm)
    (p_tp : p ≫ U.tp = cyl.π.app Γ ≫ A) (δ0_p : cyl.δ0.app Γ ≫ p = a0)
    (δ1_p : cyl.δ1.app Γ ≫ p = a1), path p p_tp ≫ U.tp =
    Path (A := A) a0 a1 (by simp [← δ0_p, p_tp]) (by simp [← δ1_p, p_tp]))
  (unpath : ∀ {Γ} {A : Γ ⟶ U.Ty} (a0 a1 : Γ ⟶ U.Tm) (a0_tp : a0 ≫ U.tp = A)
    (a1_tp : a1 ≫ U.tp = A) (p : Γ ⟶ U.Tm), p ≫ U.tp =
    Path a0 a1 a0_tp a1_tp → (cyl.I.obj Γ ⟶ U.Tm))
  (unpath_tp : ∀ {Γ} {A : Γ ⟶ U.Ty} (a0 a1 : Γ ⟶ U.Tm) (a0_tp : a0 ≫ U.tp = A)
    (a1_tp : a1 ≫ U.tp = A) (p : Γ ⟶ U.Tm) (p_tp : p ≫ U.tp = Path a0 a1 a0_tp a1_tp),
    unpath a0 a1 a0_tp a1_tp p p_tp ≫ U.tp = cyl.π.app _ ≫ A)
  (δ0_unpath : ∀ {Γ} {A : Γ ⟶ U.Ty} (a0 a1 : Γ ⟶ U.Tm) (a0_tp : a0 ≫ U.tp = A)
    (a1_tp : a1 ≫ U.tp = A) (p : Γ ⟶ U.Tm) (p_tp : p ≫ U.tp = Path a0 a1 a0_tp a1_tp),
    cyl.δ0.app _ ≫ unpath a0 a1 a0_tp a1_tp p p_tp = a0)
  (δ1_unpath : ∀ {Γ} {A : Γ ⟶ U.Ty} (a0 a1 : Γ ⟶ U.Tm) (a0_tp : a0 ≫ U.tp = A)
    (a1_tp : a1 ≫ U.tp = A) (p : Γ ⟶ U.Tm) (p_tp : p ≫ U.tp = Path a0 a1 a0_tp a1_tp),
    cyl.δ1.app _ ≫ unpath a0 a1 a0_tp a1_tp p p_tp = a1)
  (unpath_path : ∀ {Γ} {A : Γ ⟶ U.Ty} (a0 a1 : Γ ⟶ U.Tm) (p : cyl.I.obj Γ ⟶ U.Tm)
    (p_tp : p ≫ U.tp = cyl.π.app Γ ≫ A) (δ0_p : cyl.δ0.app Γ ≫ p = a0)
    (δ1_p : cyl.δ1.app Γ ≫ p = a1),
    unpath (A := A) a0 a1 (by simp [← δ0_p, p_tp]) (by simp [← δ1_p, p_tp])
    (path p p_tp) (path_tp a0 a1 p p_tp δ0_p δ1_p) = p)
  (path_unpath : ∀ {Γ} {A : Γ ⟶ U.Ty} (a0 a1 : Γ ⟶ U.Tm) (a0_tp : a0 ≫ U.tp = A)
    (a1_tp : a1 ≫ U.tp = A) (p : Γ ⟶ U.Tm) (p_tp : p ≫ U.tp = Path a0 a1 a0_tp a1_tp),
    path (A := A) (unpath a0 a1 a0_tp a1_tp p p_tp)
    (unpath_tp ..) = p)
\end{lstlisting}

The fields \lstinline{Path} and \lstinline{path} respectively provide
the formation and introduction rules for Path types,
corresponding respectively to the bottom and top arrows in the pullback diagram
\ref{diag:pathtypes} for Definition \ref{def:pathtypes}.
Both \lstinline{Path} and \lstinline{path} are stable under
substitution by \lstinline{Path_comp} and \lstinline{path_comp}.
The typing judgment for \lstinline{path} is \lstinline{path_tp}.

A Hurewicz fibration stated in terms of a
\lstinline{Cylinder} is then given.
The field \lstinline{lift} provides a chosen
diagonal lift,
with commutativity given by equations \lstinline{lift_comp_self}
and \lstinline{δ0_comp_lift}.

\begin{lstlisting}
structure Hurewicz {X Y : Ctx} (f : Y ⟶ X) where
  (lift : ∀ {Γ} (y : Γ ⟶ Y) (p : cyl.I.obj Γ ⟶ X), y ≫ f = cyl.δ0.app Γ ≫ p →
    (cyl.I.obj Γ ⟶ Y))
  (lift_comp_self : ∀ {Γ} (y : Γ ⟶ Y) (p : cyl.I.obj Γ ⟶ X)
    (comm_sq : y ≫ f = cyl.δ0.app Γ ≫ p), lift y p comm_sq ≫ f = p)
  (δ0_comp_lift : ∀ {Γ} (y : Γ ⟶ Y) (p : cyl.I.obj Γ ⟶ X)
    (comm_sq : y ≫ f = cyl.δ0.app Γ ≫ p), cyl.δ0.app Γ ≫ lift y p comm_sq = y)
\end{lstlisting}

A Hurewitz lifting structure is normal when
lifting a constant path produces a constant path.

\[\begin{tikzcd}
	Z && Y \\
	{\mathcal{I} Z} & Z & X
	\arrow["y", from=1-1, to=1-3]
	\arrow["{\delta_0}"', from=1-1, to=2-1]
	\arrow[from=1-3, to=2-3]
	\arrow[shorten >= 5pt, dashed, from=2-1, to=1-3]
	\arrow["\pi"', from=2-1, to=2-2]
	\arrow["y"{description}, from=2-2, to=1-3]
	\arrow["x"', from=2-2, to=2-3]
\end{tikzcd}\]

\begin{lstlisting}
class Hurewicz.IsNormal : Prop where
  (isNormal : ∀ {Γ} (y : Γ ⟶ Y) (p : cyl.I.obj Γ ⟶ X)
    (comm_sq : y ≫ f = cyl.δ0.app Γ ≫ p) (x : Γ ⟶ X),
    p = cyl.π.app Γ ≫ x → hrwcz.lift y p comm_sq = cyl.π.app Γ ≫ y)
\end{lstlisting}

As mentioned above,
we must also consider uniformity conditions on the lifting structure.
This means that for any map $\sigma : \Delta \to \Gamma$,
the path lifting over $\Gamma$ commutes with path lifting over $\Delta$.

\[\begin{tikzcd}
	\Delta & \Gamma & Y \\
	{\mathcal{I} \, \Delta} & {\mathcal{I} \, \Gamma} & X
	\arrow["\sigma", from=1-1, to=1-2]
	\arrow["{\delta_0}"', from=1-1, to=2-1]
	\arrow["y", from=1-2, to=1-3]
	\arrow[from=1-2, to=2-2]
	\arrow[from=1-3, to=2-3]
	\arrow[shorten >= 5pt, dashed, from=2-1, to=1-3]
	\arrow["{\mathcal{I} \sigma}"', from=2-1, to=2-2]
	\arrow[dashed, from=2-2, to=1-3]
	\arrow["x"', from=2-2, to=2-3]
\end{tikzcd}\]

\begin{lstlisting}
class Hurewicz.IsUniform : Prop where
  (lift_comp : ∀ {Γ Δ} (σ : Δ ⟶ Γ) (y : Γ ⟶ Y)
    (p : cyl.I.obj Γ ⟶ X) (comm_sq : y ≫ f = cyl.δ0.app Γ ≫ p),
    hrwcz.lift (σ ≫ y) (cyl.I.map σ ≫ p) (by simp [comm_sq, δ0_naturality_assoc]) = cyl.I.map σ ≫ hrwcz.lift y p comm_sq)
\end{lstlisting}

Then Proposition \ref{prop:Id-Elim} can be stated as follows.

\begin{lstlisting}
def polymorphicIdElim (P0 : PathType cyl U0) (hrwcz0 : Hurewicz cyl U0.tp) [hrwcz0.IsUniform]
    [hrwcz0.IsNormal] (U1 : UnstructuredUniverse Ctx) (hrwcz1 : Hurewicz cyl U1.tp)
    [Hurewicz.IsUniform hrwcz1] [Hurewicz.IsNormal hrwcz1] :
    PolymorphicIdElim (polymorphicIdIntro P0) U1 where
  ...
\end{lstlisting}

This says assuming two universes \lstinline{U0} and \lstinline{U1}
both equipped with normal uniform Hurewicz lifting structures,
a Path type structure on \lstinline{U0} gives rise to an
intensional Identity type structure on \lstinline{U0}
with elimination into types in \lstinline{U1}.
Intensional Identity introduction and elimination are formulated as follows.

\begin{lstlisting}
structure PolymorphicIdIntro (U0 U1 : UnstructuredUniverse Ctx) where
  (Id : ∀ {Γ} {A : Γ ⟶ U0.Ty} (a0 a1 : Γ ⟶ U0.Tm), (a0 ≫ U0.tp = A) → a1 ≫ U0.tp = A →
    (Γ ⟶ U1.Ty))
  (Id_comp : ∀ {Γ Δ} (σ : Δ ⟶ Γ) {A : Γ ⟶ U0.Ty} (a0 a1 : Γ ⟶ U0.Tm)
    (a0_tp : a0 ≫ U0.tp = A) (a1_tp : a1 ≫ U0.tp = A),
    Id (A := σ ≫ A) (σ ≫ a0) (σ ≫ a1) (by cat_disch) (by cat_disch) = σ ≫ Id a0 a1 a0_tp a1_tp)
  (refl : ∀ {Γ} {A : Γ ⟶ U0.Ty} (a : Γ ⟶ U0.Tm), a ≫ U0.tp = A → (Γ ⟶ U1.Tm))
  (refl_comp : ∀ {Γ Δ} (σ : Δ ⟶ Γ) {A : Γ ⟶ U0.Ty} (a : Γ ⟶ U0.Tm)
    (a_tp : a ≫ U0.tp = A), refl (σ ≫ a) (by cat_disch) = σ ≫ refl a a_tp)
  (refl_tp : ∀ {Γ} {A : Γ ⟶ U0.Ty} (a : Γ ⟶ U0.Tm) (a_tp : a ≫ U0.tp = A),
    refl a a_tp ≫ U1.tp = Id a a a_tp a_tp)

structure PolymorphicIdElim (U2 : UnstructuredUniverse Ctx) where
  (jElim : ∀ {Γ} {A : Γ ⟶ U0.Ty} (a : Γ ⟶ U0.Tm) (a_tp : a ≫ U0.tp = A)
    (C : i.motiveCtx a a_tp ⟶ U2.Ty)
    (c : Γ ⟶ U2.Tm), (c ≫ U2.tp = i.reflInst a a_tp ≫ C) →
    (i.motiveCtx a a_tp ⟶ U2.Tm))
  (jElim_comp : ∀ {Γ Δ} (σ : Δ ⟶ Γ) {A : Γ ⟶ U0.Ty} (a : Γ ⟶ U0.Tm) (a_tp : a ≫ U0.tp = A)
    (C : i.motiveCtx a a_tp ⟶ U2.Ty)
    (c : Γ ⟶ U2.Tm) (c_tp : c ≫ U2.tp = i.reflInst a a_tp ≫ C),
    jElim (σ ≫ a) (by simp [a_tp]) (i.motiveSubst σ a a_tp rfl ≫ C) (σ ≫ c)
      (by simp [*, reflInst_comp_motiveSubst_assoc]) =
    i.motiveSubst σ a a_tp rfl ≫ jElim a a_tp C c c_tp)
  (jElim_tp : ∀ {Γ} {A : Γ ⟶ U0.Ty} (a : Γ ⟶ U0.Tm) (a_tp : a ≫ U0.tp = A)
    (C : i.motiveCtx a a_tp ⟶ U2.Ty)
    (c : Γ ⟶ U2.Tm) (c_tp : c ≫ U2.tp = i.reflInst a a_tp ≫ C),
    jElim a a_tp C c c_tp ≫ U2.tp = C)
  (reflSubst_jElim : ∀ {Γ} {A : Γ ⟶ U0.Ty} (a : Γ ⟶ U0.Tm) (a_tp : a ≫ U0.tp = A)
    (C : i.motiveCtx a a_tp ⟶ U2.Ty)
    (c : Γ ⟶ U2.Tm) (c_tp : c ≫ U2.tp = i.reflInst a a_tp ≫ C),
    i.reflInst a a_tp ≫ jElim a a_tp C c c_tp = c)
\end{lstlisting}

\end{document}